\documentclass[12pt]{article}
\usepackage{epsfig}
\usepackage{amssymb}
\usepackage{amsthm}
\usepackage{amsmath}
\usepackage{amsfonts}
\usepackage{setspace}
\usepackage{enumerate}
\usepackage{epsfig}
\usepackage{subfigure}
\usepackage{color}

\newtheorem{theorem}{Theorem}[section]

\newtheorem{proposition}{Proposition}[section]
\theoremstyle{definition}
\newtheorem{definition}[theorem]{Definition}
\theoremstyle{remark}
\newtheorem{remark}[theorem]{Remark}

\numberwithin{equation}{section}

\newcommand{\be}{\begin{equation}}
\newcommand{\ee}{\end{equation}}
\newcommand{\ben}{\begin{equation*}}
\newcommand{\een}{\end{equation*}}
\newcommand{\bc}{\begin{center}}
\newcommand{\ec}{\end{center}}
\newcommand{\bal}{\begin{equation*}}
\newcommand{\enal}{\end{equation*}}
\newcommand{\al}{\alpha}
\newcommand{\bt}{\beta}
\newcommand{\gm}{\gamma}

\newcommand{\om}{\omega}

\newcommand{\nno}{\nonumber}

\begin{document}
\baselineskip=22pt

\title{\textbf{Multiresolution Analysis Based on Coalescence Hidden-variable FIF}}
\date{}
\author{G. P. Kapoor$^1$  and  Srijanani Anurag Prasad$^2$}
\maketitle \vspace{-1.5cm} \bc Department of Mathematics and Statistics\\
Indian Institute of Technology Kanpur\\  Kanpur 208016, India\\
$^1$ gp@iitk.ac.in  $^2$ jana@iitk.ac.in\ec

\onehalfspacing
\begin{abstract}
In the present paper, multiresolution analysis arising from
Coalescence Hidden-variable Fractal Interpolation Functions (CHFIFs)
is accomplished. The availability of a larger set of free variables
and constrained variables with CHFIF in multiresolution analysis
based on CHFIFs provides more control in reconstruction of functions
in $L_2(\mathbb{R})$ than that provided by multiresolution analysis
based only on Affine Fractal Interpolation Functions (AFIFs). In our
approach, the vector space of CHFIFs is introduced, its dimension is
determined and Riesz bases of vector subspaces $\mathbb{V}_k, k \in
\mathbb{Z}$, consisting of certain CHFIFs in $L_2(\mathbb{R})\bigcap
C_0(\mathbb{R})$ are constructed. As a special case, for the vector
space of CHFIFs of dimension $4$, orthogonal bases for the vector
subspaces $\mathbb{V}_k, k \in \mathbb{Z}$, are explicitly
constructed and, using these bases, compactly supported continuous
orthonormal wavelets are generated.
\end{abstract}



\newpage

\section{Introduction}

The theory of multiresolution analysis provides a powerful method to
construct wavelets having far reaching applications in analyzing
signals and images~\cite{mallat89_2,telesca04}. They permit
efficient representation of functions at multiple levels of detail,
i.e. a function $f \in L_2(\mathbb{R})$, the space of real valued
functions $g$ satisfying $\|g\|_{L^2} = \int\limits_{\mathbb{R}}
|g(x)|^2 dx < \infty $,  could be written as limit of successive
approximations, each of which is smoothed version of $f$. The
multiresolution analysis was first introduced by
Mallat~\cite{mallat89_1} and Meyer~\cite{meyer89} using a single
function. The multiresolution analysis based upon several functions
was developed in~\cite{goodman94,grochenig92,herve94}.
In~\cite{hardin92}, multiresolution analysis of $L^2(\mathbb{R})$
were generated from certain classes of Affine Fractal Interpolation
Functions (AFIFs). Such results were then generalized to several
dimensions in~\cite{geronimo93} and~\cite{geronimo92}.
In~\cite{geronimo94}, orthonormal basis for the vector space  of
AFIFs were explicitly constructed. A few years later, Donovan et
al~\cite{donovan96} constructed orthogonal compactly supported
continuous wavelets using multiresolution analysis arising from
AFIFs. The interrelations among AFIFs, Multiresolution Analysis and Wavelets are treated in detail by Massopust~\cite{massopust10} .
However, multiresolution analysis of $L_2(\mathbb{R})$ based
on Coalescence Hidden-variable Fractal Interpolation Functions
(CHFIFs) which exhibits both self-affine and non-self-affine nature
has hitherto remained unexplored. In the present work, such a
multiresolution analysis is accomplished. The availability of a
larger set of free variables and constrained variables in
multiresolution analysis based on CHFIFs provides more control in
reconstruction of functions in $L_2(\mathbb{R})$ than that provided
by multiresolution analysis based only on affine FIFs. Further,
orthogonal bases consisting of dilations and translations of scaling
functions, for the vector subspaces $\mathbb{V}_k, k \in
\mathbb{Z}$, are explicitly constructed and, using these bases,
compactly supported continuous orthonormal wavelets are generated in
the present work.

The organization of the paper is as follows: In
Section~\ref{sec:PAR}, a brief introduction on the construction of
CHFIF  is given, the vector space of CHFIFs  is introduced and a few
auxiliary results, including a result on determination of dimension
of this vector space, are found. In Section~\ref{sec:MRA}, Riesz
bases of vector subspaces $\mathbb{V}_k, k \in \mathbb{Z}$,
consisting of certain CHFIFs in $L_2(\mathbb{R})\bigcap
C_0(\mathbb{R})$ are constructed. The multiresolution analysis of
$L_2(\mathbb{R})$ is then carried out in terms of nested sequences
of vector subspaces $\mathbb{V}_k, k \in \mathbb{Z}$. As a special
case, for the vector space of CHFIFs of dimension $4$, orthogonal
bases for the vector subspaces $\mathbb{V}_k, k \in \mathbb{Z}$, are
explicitly constructed in Section~\ref{sec:OSF} and, using these
bases, compactly supported continuous orthonormal wavelets are
generated in the same section.

\section{Preliminaries and Auxiliary Results }\label{sec:PAR}

In this section, first a brief introduction on the construction of
CHFIF is given. This  is followed by the development of some
auxiliary results needed for the multiresolution analysis generated
by CHFIFs.

A Coalescence Hidden-variable Fractal Interpolation Function (CHFIF)
is constructed such that the graph of CHFIF is  attractor of an IFS.
Let the interpolation data be $\{(x_i,y_i) \in \mathbb{R}^2 :
i=0,1,\ldots,N \}$, where $-\infty < x_0 < x_1 < \ldots < x_N <
\infty$. By introducing a set of real parameters $z_i$ for
$i=0,1,\ldots,N$, consider the generalized interpolation data
$\{(x_i,y_i,z_i) \in \mathbb{R}^3 : i=0,1,\ldots,N \}$. The
contractive homeomorphisms $ L_n : I \rightarrow I_n $  for $ n =
1,\ldots, N $, are  defined by
\begin{equation}\label{eq:Ln}
L_n(x) = a_n x + b_n
\end{equation}
where, $ a_n = \frac{x_n-x_{n-1}}{x_N-x_0}$ and $ b_n = \frac{x_N
x_{n-1}-x_0 x_n}{x_N-x_0} $.  For $n=1,\ldots, N$, define the maps $
\om_n: I \times \mathbb{R}^2 \rightarrow I \times \mathbb{R}^2$ by
\begin{equation}\label{eq:omn}
\om_n(x,y,z)=(L_n(x),F_n(x,y,z))
\end{equation}
where, the functions $ F_n : I \times \mathbb{R}^2 \rightarrow
\mathbb{R}^2  $ given by,
\begin{equation}\label{eq:Fnxyz}
F_n(x,y,z)=\big(\al_n y + \bt_n z + p_n(x), \gm_n z + q_n(x)\big)
\end{equation}
satisfy  the join-up conditions
\begin{equation} \label{eq:Fncond}
F_n(x_0,y_0,z_0)=(y_{n-1},z_{n-1}) \quad \mbox{and} \quad
F_n(x_N,y_N,z_N)=(y_{n},z_{n}).
\end{equation}
Here, the variables $\al_n$,  $\gm_n$  are free variables and
$\bt_n$ are constrained variables such that $|\al_n| < 1 , \ |\gm_n|
< 1 \ \ \mbox{and}\ \ |\bt_n|+|\gm_n| < 1 $ and the functions $
p_n(x)$ and $ q_n(x)$  are linear polynomials given by
\begin{equation}\label{eq:pnqnx}
 p_n(x)=c_n x + d_n \quad \mbox{and} \quad q_n(x)=e_n x + h_n.
 \end{equation}

It is proved in~\cite{chand07} that there exist a metric equivalent
to Euclidean metric such that the functions $\om_n$, defined by $
\om _n(x,y,z)= \left( L_n(x), F_n(x,y,z)\right)$, are contraction
maps and, consequently,
\begin{align}\label{eq:chifs}
\{I \times \mathbb{R}^2; \om _n , n=1,2,\ldots, N \}
\end{align}
is the desired IFS for construction of CHFIF. Hence, there exists an
attractor $ A$ in $H(I \times \mathbb{R}^2) $ such that $ A =
\bigcup\limits_{n=1}^N \om_n(A)= \bigcup\limits_{n=1}^N \{
\om_n(x,y,z) : (x,y,z) \in A \}$ and is graph of a continuous
function $f:I \rightarrow \mathbb{R}^2$ such that $f(x_i)=(y_i,z_i)$
for $i=0,1,\ldots, N$, i.e. $A=\{(x,f(x)):x \in I ,\
f(x)=(y(x),z(x))\}$. By expressing  $ f $ component-wise as $f
=(f_1,f_2)$, \emph{Coalescence Hidden-variable Fractal Interpolation
Function} (CHFIF)~\cite{chand07} is defined as the continuous
function $f_1$ for the given interpolation data $\{(x_i,y_i):i =
0,1,\ldots,N\}$.

In order to develop the multiresolution analysis of
$L^2(\mathbb{R})$, based on CHFIF, the space of CHFIF needs to be
introduced. For this purpose, let
\begin{equation}\label{eq:tn}
t_n = (p_n,q_n)
\end{equation}
  and  $t = (t_1, t_2,\ldots,t_{N})$, where $p_n$ and $q_n$ are linear polynomials
given by~\eqref{eq:pnqnx}. Then, $ T = \{ t=(t_1,\ldots,t_{N}) : t_i
= (p_i,q_i), p_i,q_i \in \mathcal{P}_1, i=1,2,\ldots,N\} $, with
usual point-wise addition and scalar multiplication, is a vector
space,  where $\mathcal{P}_1$ is the class of linear polynomials. It
is easily seen that on $\mathbb{B}(I,\mathbb{R}^2) $, the set of
bounded functions from $I$ to $\mathbb{R}^2$ with respect to maximum
metric $d^*(f,g) = \max\limits_{x \in I } \{ |f_1(x) -
g_1(x)|,|f_2(x) - g_2(x)|\} $, the function $\Phi_{t}$ defined by
\begin{equation}\label{eq:phit}
(\Phi_{t}f)(x) = F_n(L_n^{-1}(x),f(L_n^{-1}(x)))
\end{equation}
 for $x \in I_n = [x_{n-1},x_n]$, $n=1,2,\ldots,N$, is a contraction map. Therefore, by
Banach contraction mapping theorem, $\Phi_{t}$ has a unique fixed
point $\ f_{t} \in \mathbb{B}(I,\mathbb{R}^2)$. By join-up
conditions~\eqref{eq:Fncond}, it follows that $f_{t} \in
\mathbb{C}(I,\mathbb{R}^2)$, the set of continuous functions  from
$I$ to $\mathbb{R}^2$. The following proposition gives the existence
of a linear isomorphism between the vector space $T$ and the vector
space $\mathbb{C}(I,\mathbb{R}^2)$.

\begin{proposition}\label{prop:LI}
The mapping $\Theta: T \rightarrow \mathbb{C}(I,\mathbb{R}^2)$
defined by $ \Theta(t)=f_{t}$ is a linear isomorphism.
\end{proposition}
\begin{proof}
 The assertion of the proposition  is proved by
establishing \\(i) $( a f_t + f_s )_i(x) = a f_{t, i}(x)+ f_{s,
i}(x) $, $i=1,2$, where $f_t$  and $a f_t + f_s$  are written
component-wise as $f_t = (f_{t, 1},f_{t, 2})$  and $ a f_t +f_s =
((a f_t +f_s)_1 , (a f_t+f_s)_2)$, (ii) $( a f_t + f_s ) = f_{a t +
s}$, (iii) $\Theta$ is onto and (iv) $\Theta$ is one-one.

The identity (i) follows by equating the components of left and
right hand side in the identity $ ( a f_{t} + f_s )(x) = a f_{t}(x)+
f_s(x)$.

(ii) $( a f_{t} + f_s ) = f_{a t + s}$  : By the definition of
$\Phi_{t}$,
\begin{align*}
\lefteqn{(\Phi_{a t + s}(a f_{t} + f_s ))(x)}\\
& = F_n(L_n^{-1}(x),(a f_{t} + f_s )(L_n^{-1}(x)))\\
& =  \Big(\al_n (a f_{t} + f_s)_1(L_n^{-1}(x)) + \bt_n (a f_{t} +
f_s)_2(L_n^{-1}(x)) + (a\ p_n + \hat{p}_n)(L_n^{-1}(x)),
\\ & \quad \quad \gm_n (a f_{t} + f_s)_2(L_n^{-1}(x)) + (a\ q_n +\hat{q}_n)(L_n^{-1}(x)) \Big)
\end{align*}
Using identity (i), it follows that,\vspace{-0.2cm}
\begin{align*}
\lefteqn{(\Phi_{a t + s}(a f_{t} + f_s ))(x)}\\
& = \Big(\al_n \big(a f_{t,1} (L_n^{-1}(x)) + f_{s,1}
(L_n^{-1}(x))\big) +
\bt_n \big(a f_{t,2} (L_n^{-1}(x)) + f_{s,2}(L_n^{-1}(x))\big)\\
& \quad \quad \mbox{}+ \big(a\ p_n (L_n^{-1}(x)) +
\hat{p}_n(L_n^{-1}(x))\big),
\\ & \quad \quad \gm_n \big(a f_{t,2} (L_n^{-1}(x)) + f_{s,2}(L_n^{-1}(x))\big) +
\big(a\  q_n (L_n^{-1}(x)) + \hat{q}_n(L_n^{-1}(x))\big) \Big)
\end{align*}
The above equation gives the following on simplification:\vspace{-0.2cm}
\begin{align*}
\lefteqn{(\Phi_{a t + s}(a f_{t} + f_s ))(x)}\\
&= a \Big(\al_n f_{t,1} (L_n^{-1}(x)) + \bt_n f_{t,2} (L_n^{-1}(x))
+ p_n (L_n^{-1}(x)),
\gm_n f_{t,2} (L_n^{-1}(x))  + q_n (L_n^{-1}(x)) \Big) \\
& \quad \mbox{}+ \Big(\al_n f_{s,1} (L_n^{-1}(x)) + \bt_n f_{s,2}
(L_n^{-1}(x)) + \hat{p}_n (L_n^{-1}(x)),
\gm_n f_{s,2} (L_n^{-1}(x))  + \hat{q}_n (L_n^{-1}(x)) \Big)\\
& = a f_{t}(x) + f_s(x)
\end{align*}
Therefore, $a f_{t} + f_s$ is a fixed point of $\Phi_{a t + s}$ for
all $ a \in \mathbb{R}$ and $ t, s \in T$. By uniqueness of fixed
point of $\Phi_{a t + s}$, it follows that $( a f_{t} + f_s ) = f_{a
t + s}$.

(iii) $\Theta$ is onto :
 Let $f =(f_1,f_2 )\in
\mathbb{C}(I,\mathbb{R}^2)$. Define $q_i(f) = f_2  \circ L_i - \gm_i
f_2 $ and $p_i(f) = f_1 \circ L_i - \al_i f_1 - \bt_i f_2 $ for
$i=1,\ldots,N$. Suppose $t (f)= (t_1(f), t_2(f),\ldots t_N(f))$,
where  $ t_i(f)=(p_i(f),q_i(f))$. Then $t(f) \in T$ whenever $f \in
\mathbb{C}(I,\mathbb{R}^2)$. Also $f_{t(f)} =f$.

(iv) $\Theta$  is one-one : Let $f_{t}(x)= (0,0) $ for all values of
$x \in I$. Then,\vspace{-0.3cm}
\begin{align*}
f_{t}(x) = (0,0) & \Leftrightarrow  \Phi_{t}(f_{t})(x) = (0,0) \\
& \Leftrightarrow  F_n(L_n^{-1}(x), f_{t}(L_n^{-1}(x))) = (0,0)\\
& \Leftrightarrow  (p_n, q_n)=(0,0) \ \mbox{for every n} \\
& \Leftrightarrow   t = (0,\ldots,0)
\end{align*}
\end{proof}

To introduce the space of CHFIFs,  let the set $\mathcal{S}_0 $
consisting of functions $ f~:~I \rightarrow \mathbb{R}^2 $  be
defined as
  $ \mathcal{S}_0 = \{ f :
f=(f_1,f_2),\ f_1 \ \mbox{is a CHFIF passing through} \\
\{(x_i,y_i) \in \mathbb{R}^2 : i=0,1,\ldots,N \} \ \mbox{and}\ f_2 \
\mbox{is an AFIF passing through   } \{(x_i,z_i) \in \mathbb{R}^2 :
i=0,1,\ldots,N \} \} $. Then, $\mathcal{S}_0 $ is a vector space,
with usual point-wise addition and scalar multiplication. The space
of CHFIFs is now defined as follows:

\begin{definition}\label{defn:chfifspace}
Let $\mathcal{S}_0^1 $ be the set of functions $ f_1 : I \rightarrow
\mathbb{R}$ that are first components  of functions $f \in
\mathcal{S}_0$. The \textbf{space of CHFIFs} is the  set
$\mathcal{S}_0^1$ together with the maximum metric $d^*(f,g) =
\max\limits_{x \in I} |f(x) - g(x)|$.
\end{definition}

It is easily seen that the space of CHFIFs $\mathcal{S}_0^1 $ is
also a vector space with point-wise addition and scalar
multiplication. The following proposition gives the dimension of
$\mathcal{S}_0^1 $:

\begin{proposition}\label{prop:dim}
The dimension of space of CHFIFs   is $2N$.
\end{proposition}

\begin{proof}
Consider the operator  $\Phi_{t}f =(\Phi_{t,1} f_1, \Phi_{t,2}
f_2)$. The operators $ \Phi_{t,i} : \mathbb{B}(I,\mathbb{R})
\rightarrow \mathbb{B}(I,\mathbb{R})$, $i=1,2$, where
$\mathbb{B}(I,\mathbb{R})$ is the set of bounded functions, satisfy
\begin{equation}\label{eq:la1}
\Phi_{t,1} f_1(x) = \al_n f_1(L_n^{-1}(x)) + \bt_n f_2(L_n^{-1}(x))
+ p_n (L_n^{-1}(x))
\end{equation} and
\begin{equation}\label{eq:la2}
\Phi_{t,2} f_2(x) =  \gm_n f_2(L_n^{-1}(x)) + q_n (L_n^{-1}(x))
\end{equation}
 for $x \in [x_{n-1},x_{n}]$. By Proposition~\ref{prop:LI} and~\eqref{eq:la2}, it follows that
$f_2$ is completely determined by $f_2(\frac{i}{N})$ for
$i=0,1,\ldots N$. Further, it follows by~\eqref{eq:la1} that $f_1$
depends on $f_2$. Then, for $f =(f_1,f_2) \in \mathcal{S}_0$,  the
function $f_1$ is the unique CHFIF passing through $ (\frac{i}{N},
y_i)$, while the function $f_2$ is the unique AFIF  passing through
$ (\frac{i}{N}, z_i)$. Hence,
\begin{equation}\label{eq:dms0}
\mbox{dimension of}\ \mathcal{S}_0 = 2(N+1).
\end{equation}
Consider the projection map $P : S_0 \rightarrow S_0^1 $. Then,
Kernel of $P = \{f \in \mathcal{S}_0 \ \mbox{such that} \ P(f) =
0\}$ is a proper subset of $S_0$ and consists of elements in the
form $(0,0)$ and $(0,f_2)$. For the element $(0,f_2) \in \mbox{Ker}
P$, it is observed that $\bt_n f_2(L_n^{-1}(x)) + p_n(L_n^{-1}(x)) =
0 $ for $x \in I_n$. Hence, for all $x \in I$, it is seen that
$f_2(x) =\frac{-1}{\bt_n} p_n(x)$. With $x=x_0$, it follows that
$c_i = \frac{\bt_i}{\bt_1}c_1$ and $d_i = \frac{\bt_i}{\bt_1}d_1,
i=2,\ldots,N$. Consequently, if $(0,f_2) \in \mbox{Ker} P $ then
$f_2$ is a linear polynomial. So, dimension of $\mbox{Ker} P = 2$.
Therefore, by Rank-Nullity Theorem [56], $ \mbox{dimension of}\
\mathcal{S}_0^1 = \mbox{dimension of}\ \mathcal{S}_0 -
\mbox{dimension of Ker}\ P = 2(N+1) - 2 = 2N$.
\end{proof}

\begin{remark}\label{rem:S0closed}
By Proposition~\ref{prop:LI} and ~\eqref{eq:dms0}, it follows that
the map $\theta: \mathbb{R}^{N+1} \times \mathbb{R}^{N+1}
\rightarrow \mathcal{S}_0$ defined by $\theta(y,z) = f $ is a linear
isomorphism, where $ f=(f_1,f_2) \in \mathcal{S}_0$, $f_1$ is the
unique CHFIF passing through the points $ (x_i, y_i)$ and $f_2$ is
the unique AFIF passing through the points $ (x_i, z_i)$,  $\ y =
(y_0,y_1,\ldots,y_N) $ and $\ z = (z_0,z_1,\ldots,z_N )$. Thus,
$\mathcal{S}_0$ is linearly isomorphic to $\mathbb{R}^{N+1} \times
\mathbb{R}^{N+1}$. Consider the metric space $(\mathbb{R}^{N+1}
\times \mathbb{R}^{N+1},d_{\mathbb{R}^{2(N+1)}} ) $, where
$d_{\mathbb{R}^{2(N+1)}}$ is given by  $d_{\mathbb{R}^{2(N+1)}}( y
\times z, \bar{y} \times \bar{z})= \max\limits_{0 \leq i \leq N}
(|y_i - \bar{y}_i|,|z_i-\bar{z}_i|)$, $\ y = (y_0,y_1,\ldots,y_N)$,
$ z = (z_0,z_1,\ldots,z_N )$,  $ \bar{y} =
(\bar{y}_0,\bar{y}_1,\ldots,\bar{y}_N) $ and $\ \bar{z} =
(\bar{z}_0,\bar{z}_1,\ldots,\bar{z}_N )$.  Then, with the metric
$d^*$ on the set $\mathcal{S}_0$, it is observed by~\eqref{eq:phit}
that the maps $\theta$ and $\theta^{-1}$ are continuous. Hence
$\mathcal{S}_0$ is closed and complete subspace of
$L_2(\mathbb{R})$.
\end{remark}

\begin{remark}\label{rem:S01closed}
Let $\{f_{n,1}\}$ be a sequence in $\mathcal{S}_0^1$ such that
$\lim\limits_{ n \rightarrow \infty} f_{n,1} = f^*_1 $ and  $\{f_n =
(f_{n,1},f_{n,2})\}$ be a convergent sequence  in $\mathcal{S}_0$,
where $f_{n,2}$  are AFIFs. Since $\mathcal{S}_0$ is closed,
$\lim\limits_{ n \rightarrow \infty} f_n = f^* \equiv(f^*_1,f^*_2)
\in \mathcal{S}_0$. Thus, $ f^*_1 \in \mathcal{S}_0^1$ and
consequently, $\mathcal{S}_0^1$ is closed subspace of
$L_2(\mathbb{R})$.
\end{remark}

\section{Multiresolution Analysis Based on CHFIF }\label{sec:MRA}

In this section, the multiresolution analysis  of $L_2(\mathbb{R})$
is generated by using a finite set of CHFIFs. For this purpose, the
sets $\mathbb{V}_k, k \in \mathbb{Z}$, consisting of collection of
CHFIFs are defined.  It is first shown that the sets $\mathbb{V}_k$
form a nested sequence. The multiresolution analysis of $
L_2(\mathbb{R}) $ is then generated by constructing Riesz bases of
vector subspaces $\mathbb{V}_k$ consisting of orthogonal functions
in $ L_2(\mathbb{R}) $.

To introduce certain sets of  CHFIFs  needed for multiresolution of
$L_2(\mathbb{R})$, let $L_2(\mathbb{R},\mathbb{R}^2) $ be a
collection of functions $f :\mathbb{R} \rightarrow \mathbb{R}^2$
such that $f=(f_1,f_2)$  and $f_1 $, $f_2 \in L_2(\mathbb{R})$ and
$C_0(\mathbb{R},\mathbb{R}^2) $ be a collection of functions $f
:\mathbb{R} \rightarrow \mathbb{R}^2$ such that $f=(f_1,f_2)$  and
$f_1 $, $f_2 \in C_0(\mathbb{R})$, the set of all real valued
continuous functions defined on $\mathbb{R}$ which vanish at
infinity.    Define the set $\tilde{\mathbb{V}}_0$  as
\begin{equation*}
 \tilde{\mathbb{V}}_0
= \tilde{S}_0 \bigcap L_2(\mathbb{R},\mathbb{R}^2) \bigcap
C_0(\mathbb{R},\mathbb{R}^2)
\end{equation*}
where, $\tilde{S}_0=\{ f : f  =(f_1,f_2),\ f_1|_{[i-1,i)} \ \mbox{is
a CHFIF  and}
  f_2|_{[i-1,i)} \ \mbox{is an AFIF, } \ i \in \mathbb{Z} \}$.

That the set $ \tilde{\mathbb{V}}_0$ is not empty is easily seen by
considering a function $f=(f_1,f_2) \in \mathcal{S}_0$, with
$f(x_0)=(0,0)=f(x_N)$ and $f(x) = (0,0) $ for $x \not \in I$, which
obviously belongs to $\tilde{\mathbb{V}}_0$. Let, for $ k \in
\mathbb{Z}$,
\begin{equation*}
 \tilde{\mathbb{V}}_k
= \{ f :   f(N^{-k} \cdot) \in \tilde{\mathbb{V}}_0\}.
\end{equation*}

\newpage
The sets $ \tilde{\mathbb{V}}_0$ and $ \tilde{\mathbb{V}}_k$ are
easily seen to be closed sets as follows: Let $\{f_n\}$ be a
sequence in $\tilde{\mathbb{V}}_0$ such that $\lim\limits_{ n
\rightarrow \infty} f_n = f^* = (f^*_1,f^*_2)$. Now, $\lim\limits_{
n \rightarrow \infty} f_n|[i-1,i) = f^*|[i-1,i) =
(f^*_1|[i-1,i),f^*_2|[i-1,i))$. By Remark~\ref{rem:S0closed} and
Remark~\ref{rem:S01closed}, it is observed that $f^*_1|[i-1,i)$ is a
CHFIF  and $f^*_2|[i-1,i)$ is an AFIF, $i \in \mathbb{Z}$. Thus,
$f^* \in \tilde{S}_0$, which implies that $\tilde{S}_0$ is a closed
set. Consequently, $\tilde{\mathbb{V}}_0$ and $\tilde{\mathbb{V}}_k,
, k \in \mathbb{Z}$ are closed sets. Now, consider sets
$\mathbb{V}_0$ and $\mathbb{V}_k, k \in \mathbb{Z}\setminus 0$,  of
CHFIFs  defined as follows:
\begin{equation}\label{eq:setV0}
 \mathbb{V}_0
= \{ f_1 : f_1 \ \mbox{ is the first component of some } f
=(f_1,f_2) \in \tilde{\mathbb{V}}_0\}
\end{equation}
and
\begin{equation}\label{eq:setVk}
 \mathbb{V}_k
= \{ f_1 :    f_1(N^{-k} \cdot) \in \mathbb{V}_0\}.
\end{equation}
The sets $\mathbb{V}_k , k \in \mathbb{Z}  $,  with $L_2$-norm, are
subspaces  of $L_2(\mathbb{R})$. The following proposition shows
that the subspaces $\mathbb{V}_k , k \in \mathbb{Z}$,  of
$L_2(\mathbb{R})$, form a nested sequence:

\begin{proposition}\label{prop:subseq1}
The subspaces $\mathbb{V}_k  , k \in \mathbb{Z}$ are vector
subspaces of $L_2(\mathbb{R})$ and form a nested sequence $\ldots
\supseteq \mathbb{V}_{-1} \supseteq \mathbb{V}_0 \supseteq
\mathbb{V}_1 \supseteq \ldots $.
\end{proposition}

\begin{proof}
It follows from Proposition~\ref{prop:LI} that the sets
$\mathbb{V}_k, k \in \mathbb{Z}$ are vector subspaces of
$L_2(\mathbb{R})$. Now, to show $\mathbb{V}_k \supseteq
\mathbb{V}_{k+1} $ for all $k \in \mathbb{Z}$, it suffices to prove
the inclusion relation for $k=0$. Let $f \in \mathbb{V}_1$. Then,
$f|_{[0,N) }=g_1|_{[0,N)}$ for some  $ g=(g_1,g_2)
\in~\tilde{\mathbb{V}}_1$. If  $ G = \mbox{graph}(g|_{[0,N]})$ then,
$G~=~\bigcup\limits_{i=1}^{N} w_i (G)$ implies, for $j \in
\{1,\ldots,N\}$,  $w_j(G) = \bigcup\limits_{i=1}^{N} w_j \circ w_i
\circ w_j^{-1}(w_j(G))$, where $ w_i(G) =(L_i(x),F_i(x,y,z)) $ for
all $(x,y,z) \in G$ , $i=1,\ldots N$. Expressing $w_i$ and $w_j
\circ w_i \circ w_j^{-1}$ in matrix form as $w_i(x,y,z) = A_i(x,y,z)
+ B_i$ and $w_j \circ w_i \circ w_j^{-1}(x,y,z) = A_{i,j}(x,y,z) +
B_{i,j}$, it is observed that non-zero entries in matrices $A_i$ and
$A_{i,j}$ occur at the same places. Consequently, $w_j(G)$ is graph
of $g|_{[j-1,j)}$, so that $g \in \tilde{\mathbb{V}}_0$. It
therefore follows that $g_1|_{[j-1,j)}$ is a CHFIF on the interval
$[j-1,j)$. Thus, the function $f|_{[j-1,j)} = g_1|_{[j-1,j)} $ is a
CHFIF on the interval $[j-1,j)$ and consequently, $f \in
\mathbb{V}_0$.  \end{proof}

In order to generate a multiresolution analysis  of
$L_2(\mathbb{R})$ using CHFIFs, the inner product on the space
$\mathbb{V}_k , k \in \mathbb{Z}  $ is defined by $ \langle
f_1,\hat{f}_1 \rangle = \int\limits_{\mathbb{R}} f_1(x)\hat{f}_1(x)\
dx $. Using the following recurrence relations, \vspace{-0.5cm}
\bc $
f_1(L_n(x)) = \al_n f_1(x) + \bt_n f_2(x) + p_n (x) $ \ec   and \vspace{-0.5cm} \bc $
\hat{f}_1(L_n(x)) = \hat{\al}_n \hat{f}_1(x) + \hat{\bt}_n \hat{f}_2(x) + \hat{p}_n (x),
$
\ec
it is observed that, for  $f_1, \hat{f}_1 \in
\mathbb{V}_0$,
\begin{equation}\label{eq:I}
\langle f_1,\hat{f}_1\rangle = \frac{ \begin{array}{r}
\sum\limits_{n=1}^{N} a_n \Big(
 \al_n\ \hat{\bt}_n\ \langle f_1,\hat{f}_2\rangle + \bt_n\ \hat{\al}_n\ \langle f_2,\hat{f}_1\rangle
 +  \bt_n\ \hat{\bt}_n\
\langle f_2,\hat{f}_2\rangle   \\ +  \al_n \langle
f_1,\hat{p}_n\rangle
 + \hat{\al}_n \langle \hat{f}_1,p_n\rangle   + \bt_n \langle
f_2,\hat{p}_n\rangle \\+ \hat{\bt}_n \langle \hat{f}_2,p_n\rangle +
\langle p_n,\hat{p}_n\rangle \Big)
\end{array}}{1- \sum\limits_{n=1}^{N} a_n \al_n
\hat{\al}_n}
\end{equation}
where, $a_n$, $\al_n$ and $\bt_n$, $p_n$;   $ \hat{\al}_n$ and  $
\hat{\bt}_n$,    $\hat{p}_n$,   are given by ~\eqref{eq:Ln},
\eqref{eq:Fnxyz}, \eqref{eq:pnqnx} respectively for the
interpolation data $\{(x_i,y_i,z_i) : i=0,1,\ldots,N\}$ and
$\{(x_i,\hat{y}_i,\hat{z}_i) : i=0,1,\ldots,N\}$.
Using~\eqref{eq:I}, the set of orthogonal functions that forms the
Riesz basis of set $\mathbb{V}_0$ is constructed as follows:

Let, the free variables $\al_j$, $\gm_j$  and constrained variables
$\bt_j$, $j=1,\ldots, N$, $N > 1$, in the construction of CHFIF be
chosen such that $\al_j + \bt_j \neq \gm_j$ for atleast one $j$.
Consider, the points $y_i$ and $z_i \in R^{n+1}$, $i=0,\ldots,N$,
given by
\begin{equation}\label{eq:ypoints}
\left. \begin{array}{lll} \ y_{0}&=(1,r_1,\ldots,r_{N-1},0), &
y_{N}=(0,s_1,\ldots,s_{N-1},1), \\ \ y_{i}&=(0,\ldots,1,\ldots,0) ,&  i=1,\ldots,N-1,  \\
y_{N+1+i}&=(0,u_{i,1},\ldots,u_{i,N-1},0), &  i=0,\ldots,N;
\end{array} \right\}
\end{equation}
\begin{equation}\label{eq:zpoints}
z_{i}=(0,\ldots,0),\   \ z_{N+1+i}=(0,\ldots,1,\ldots,0) , \
i=0,\ldots,N
\end{equation}
and a set of $2(N+1)$ functions $\tilde{f}_i =
(\tilde{f}_{i,1},\tilde{f}_{i,2}) \in \mathcal{S}_0, \ i=0, \ldots,
2N+1$, where the CHFIF $\tilde{f}_{i,1}$ passes through the points
$(x_k,y_{i_k}), k=0,\ldots,N+1, \ y_{i_k}$ being the $k^{th}$
component of $y_i$ and AFIF $\tilde{f}_{i,2}$ passes through the
points  $(x_k,z_{i_k}), k=0,\ldots,N+1, \ z_{i_k}$ being the
$k^{th}$ component of $z_i$. Let the function $\tilde{f}^*_i :
\mathbb{R} \rightarrow \mathbb{R}^2$, $i=0,1,\ldots 2N+1$,  be the
extension of the function $\tilde{f}_i : I \rightarrow \mathbb{R}^2$
such that $\tilde{f}^*_i(x) = \tilde{f}_i(x) $ for $ x \in I$ and
$\tilde{f}^*_i(x) = (0,0)$ for $ x  \not \in  I$.

For ensuring the orthogonality of the functions $\tilde{f}^*_{i,1}$
with respect to the inner product  in $L_2(\mathbb{R})$, let the
values of  $r_i $ , $s_i $ and $u_{i,j}$, $ i,j=1,\ldots,N-1$,
in~\eqref{eq:ypoints} be chosen such that
\begin{equation}\label{eq:risiuij}
  \langle\tilde{f}_{i,1}, \tilde{f}_{0,1}\rangle  =  0  ,
 \quad \ \langle \tilde{f}_{i,1},\tilde{f}_{N,1}\rangle = 0, \quad
 \langle\tilde{f}_{N+1+j,1},\tilde{f}_{i,1}\rangle  =  0.
\end{equation}

Let, for $ i=1,2,\ldots,N-1$,
\begin{equation}\label{eq:zetaeta}
 \zeta_i  = < \tilde{f}_{N+1+i,1},\tilde{f}_{0,1}>  \quad \mbox{and}
 \quad
\eta_i = <  \tilde{f}_{N+1+i,1},\tilde{f}_{N,1}>.
\end{equation}
The free variables $\al_j, \gm_j$  and constrained variables $
\bt_j$, $ j=1,2,\ldots,N$, in~\eqref{eq:Fnxyz} are $3N$ variables
and  $\zeta_i =\eta_i =0$, $ i=1,\ldots, N-1$ is a system of $2N-2$
equations. Suppose there exist no $\alpha_j, \gamma_j $ and
$\beta_j$, $j=1,\ldots,N$, in $(-1,1)$ such that $\zeta_i =\eta_i
=0$, $ i=1,\ldots, N-1$, then dimension of $S_0^1 < 2N$, which is a
contradiction. Hence, there exists atleast one set of $\alpha_j,
\gamma_j $ and $\beta_j$, $j=1,\ldots, N$, in $(-1,1)$ such that
$\zeta_i =\eta_i =0$, $ i=1,\ldots, N-1$. The free variables $\al_j,
\gm_j$  and constrained variables $ \bt_j$, $ j=1,2,\ldots,N$,
in~\eqref{eq:Fnxyz} are chosen such that, for $ i=1,2,\ldots,N-1$,
$\zeta_i =0$ and $\eta_i =0$.

It is easily seen that the functions $\tilde{f}^*_i$, $\
i=0,\ldots,2N+1$, $\ \tilde{f}^*_{i,1}$, $\ i=0,\ldots,N\ $  and the
functions $\ \tilde{f}^*_{j,2}$,  $j=N+1,\ldots,2N+1$, are linearly
independent. Now, by~\eqref{eq:phit},  $\ \tilde{f}^*_{j,1}(x) =
\al_n \tilde{f}^*_{j,1}(L_n^{-1}(x)) + \bt_n
\tilde{f}^*_{j,2}(L_n^{-1}(x)) + p_{n,j} (L_n^{-1}(x))$  and  $
\tilde{f}^*_{j,2}(x) = \gm_n \tilde{f}^*_{j,2}(L_n^{-1}(x)) +
q_{n,j} (L_n^{-1}(x))$,  $\ j=0,1,\ldots, 2N+1$, where $p_{n,j}$ and
$q_{n,j}$  are linear polynomials. By~\eqref{eq:zpoints}, the
functions $\tilde{f}^*_{j,2}$, $\ j=N+2,\ldots,2N$, are not linear
polynomials. Hence, $\sum\limits_{k=1}^{N-1} a_k
\tilde{f}^*_{N+1+k,1}(x) - \al_n \sum\limits_{k=1}^{N-1} a_k
\tilde{f}^*_{N+1+k,1}(L_n^{-1}(x)) = 0$  if and only if  $a_k=0$,
which implies $\tilde{f}^*_{N+1+k,1}$, $k=1,\ldots,N-1$, are
linearly independent. The linear independence of $\
\tilde{f}^*_{k,1}$, $\tilde{f}^*_{N+1+k,1}$, $k=1,\ldots,N-1$
together with~\eqref{eq:risiuij} will yield same number of
orthogonal functions in Gram-Schmidt process.

Let $\{\phi_{i,1}\}_{i=1}^{2N-1} \subset\mathbb{V}_0 , i \neq N, $
be a sequence of orthogonal functions obtained from
$\{\tilde{f}^*_{i,1}\}_{i=1}^{2N}, i \neq N,N+1$   by the
Gram-Schmidt process. Set, \vspace{-0.3cm}
\begin{equation*}  \phi_{N,1} =   \left \{
\begin{array}{ll}
\tilde{f}^*_{N,1}(x)& x \in [0,1) \\
 \tilde{f}^*_{0,1}(x-1)& x \in [1,2) \\
 0 & \mbox{otherwise}
 \end{array}
\right.
\end{equation*}
It is easily seen by Proposition~\ref{prop:dim} that $\phi_{i,1}$,
$i=1,2,\ldots,2N-1$, are non-zero functions. Further,
by~\eqref{eq:risiuij} and~\eqref{eq:zetaeta}, it follows that
$\{\phi_{i,1}: i=1,2,\ldots,2N-1\}$ is an orthogonal set. This is
the set required for the generation of multiresolution analysis of
$L_2(\mathbb{R})$ in the following theorem :

\begin{theorem}\label{th:mra}
Let free variables $\al_j$, $\gm_j$  and constrained variables
$\bt_j$, $j=1,\ldots, N$, $N > 1$, in the construction of CHFIF be
chosen such that $\al_j + \bt_j \neq \gm_j$ for atleast one $j$  and
$\zeta_i  $ ,  $\eta_i $ given by~\eqref{eq:zetaeta} be such that
$\zeta_i =0 $ ,  $\eta_i =0 $,  $i=1,\ldots,N-1$. Then,
\begin{equation}\label{eq:V0clos}
 \mathbb{V}_0 =
\mbox{clos}_{L^2}\ \mbox{span} \{\phi_{i,1}(\cdot -l):
i=1,\ldots,2N-1, \  l \in \mathbb{Z} \},
\end{equation}
where, $\phi_{i,1} \in \mathbb{V}_0$. Also, the set
$\{\phi_{i,1}\}_{i=1}^{2N-1}$ generates a continuous, compactly
supported multiresolution analysis of $L_2(\mathbb{R})$.
\end{theorem}

\begin{proof}
It is obvious that functions $\phi_{i,1} ,i=1, \ldots,2N-1$, are
compactly supported and are elements of $\mathbb{V}_0$. Now, $\ f
\in \mathbb{V}_0 $  implies  $ f|_{[i-1,i) }=g_1|_{[i-1,i) } \ $ is
CHFIF for some $ g=(g_1,g_2) \in \tilde{\mathbb{V}}_0$. Since, every
$ g=(g_1,g_2) \in \tilde{\mathbb{V}}_0$ is determined by  $
g_1(\frac{i}{N})$ and $g_2(\frac{i}{N})$, $i \in \mathbb{Z}$, the
function  $g$ has unique expansion in terms of the functions
$\tilde{f}^*_{i}= (\tilde{f}^*_{i,1},\tilde{f}^*_{i,2})$,
$i=0,\ldots 2N+1$, and their integer translates.  Hence, the
function $f=g_1 \in \mathbb{V}_0$ has a unique expansion in terms of
the functions  $\tilde{f}^*_{i,1}$,   $i=1,\ldots N-1,
N+2,\ldots,2N-2, \ \phi_{N,1}$ and their integer translates. Thus,
CHFIF $f$ has a unique expansion in terms of the functions
$\phi_{i,1} ,\ i=1, \ldots,2N-1$, and their integer translates, i.e.
$ f = \sum\limits_k
 \bigg(\sum\limits_{i=1}^{2N-1}  K_{k,i}  \phi_{i,1}(x-k) \bigg)$
 where,  $K_{k,i} = \int\limits_{\mathbb{R}} f(x)\phi_{i,1}(x-k) \
dx $. Since $f \in \mathbb{V}_0$ is arbitrary, $\mathbb{V}_0 = \
\mbox{span} \{\phi_{i,1}(\cdot -l), \ i=1,\ldots,2N~-~1,
l~\in~\mathbb{Z} \}$. Let $\{f_{n,1}\}$ be a sequence in
$\mathbb{V}_0$ such that $\lim\limits_{ n \rightarrow \infty}
f_{n,1} = f^*_{1} $ and $\{f_n = (f_{n,1},f_{n,2})\}$ be a
convergent sequence  in $\mathbb{V}_0$, where $f_{n,2}$  are AFIFs.
Since $\tilde{\mathbb{V}}_0$ is closed, $\lim\limits_{ n \rightarrow
\infty} f_n = f^* =(f^*_1,f^*_2) \in \tilde{\mathbb{V}}_0$ which
gives $f^*_1|_{[i-1,i)},  i \in \mathbb{Z}$ is a CHFIF. Hence, $
f^*_1 \in \mathbb{V}_0$. Therefore, $\mathbb{V}_0$ is closed and $
\mathbb{V}_0 = \mbox{clos}_{L^2}\ \mbox{span} \{\phi_{i,1}(\cdot
-l), \ i=1,\ldots,2N-1, \  l \in \mathbb{Z} \}$.

Now, it is shown that the set $\{\phi_{i,1}\}_{i=1}^{2N-1}$
generates a continuous, compactly supported multiresolution analysis
of $L_2(\mathbb{R})$:

 (a) By Proposition~\ref{prop:subseq1}, it follows that
 $\ldots \supseteq \mathbb{V}_{-1} \supseteq \mathbb{V}_0 \supseteq \mathbb{V}_1 \supseteq
\ldots $.

 (b)To prove that  $  \bigcap\limits_{k \in \mathbb{Z}} \mathbb{V}_k = \{0\}$, let $ I_n =[ n, n+1], n \in
 \mathbb{Z}$, and $U_0 = \{f_{\chi_{I_0}} : f \in \mathbb{V}_0\}$ where, $f_{\chi_{I_0}} (x) = f (x) $
 if $x \in I_0$ and $f_{\chi_{I_0}} (x) = 0 $
 if $x \not \in I_0$. Since the space $U_0$ is finite dimensional
 over $\mathbb{R}$,  the norms  $\| \cdot \|_{\infty}$ and $\| \cdot
 \|_{L^2}$  restricted to $U_0$ are equivalent.
 Hence, there exist a positive constant $c$
 such that $\| f \|_{\infty} \leq c \| f
 \|_{L^2}$ for all $f \in U_0$. By the property of translation invariance,
it is observed that $\| f_{\chi_{I_n}} \|_{\infty} \leq c \|
f_{\chi_{I_n}}
 \|_{L^2}$ for any $ f \in U_0$. Thus,
 $\| f \|_{\infty} \leq \sup\limits_n \| f_{\chi_{I_n}} \|_{\infty}
 \leq  c  \sum\limits_{n \in \mathbb{Z}} \| f_{\chi_{I_n}} \|_{L^2} = c \|   f  \|_{L^2}$
for any $ f \in \mathbb{V}_0$. It therefore follows by the
definition of $\mathbb{V}_k$ that $\| f \|_{\infty} \leq c N^{k/2}\|
f  \|_{L^2}$ for all $f \in \mathbb{V}_k$. Consequently, if $ f \in
\bigcap\limits_{k \in \mathbb{Z}} \mathbb{V}_k$, then $\| f
\|_{\infty}=0$ which implies $f=0$.

 (c) For showing that $\mbox{clos}_{L_2}\ \bigcup\limits_{m \in \mathbb{Z}} \mathbb{V}_m =
L^2(\mathbb{R})$, let $f=(f_1,f_2) \in \tilde{\mathbb{V}}_0$, where
$f_1$ is a CHFIF passing through $(x_k,1)$ and $f_2$ is an AFIF
passing through $(x_k,z_k)$. For all $x \in \mathbb{R}$,
by~\eqref{eq:ypoints},
\begin{equation*}
f_1 = \sum\limits_k
 \bigg(\sum\limits_{i=1}^{N-1}  C_{k,i}
 \tilde{f}^*_{i,1} (x-k) +  \phi_N(x-k) + \sum\limits_{i=1}^{N-1}
 D_{k,i} \tilde{f}^*_{N+1+i,1} (x-k)\bigg)
 \end{equation*}
 where,
\begin{equation*}
 C_{k,i} =
\Big(1-r_i-s_i-\sum\limits_{j=1}^{N-1}u_{j,i}z_j\Big) 
\quad \mbox{and} \quad D_{k,i}  =z_i,  \quad i=1,\ldots,N-1.
\end{equation*}
Now, since $\phi_{i, 1}$ are continuous and compactly supported, by
using (b) and Proposition 3.1 of~\cite{geronimo94},  it follows that
$\bigcup\limits_{k \in \mathbb{Z} } \mathbb{V}_k$ is dense  in
$L_2(\mathbb{R})$.

 (d) For proving that the functions $\phi_{i,1}$,  $i=1, \ldots, 2N-1$, and
 their integer translates form a Riesz basis for $\mathbb{V}_0$, let
$\tau$ be the smallest eigenvalue of the matrix
\begin{equation*}
 \frac{1}{
\|\phi_{N,1}\|^2 }  \left(\begin{array}{ll} \left(\int_I |f_{0,1}
(x)|^2 dx \right)^{1/2} & \
 \left(\int_I |f_{0,1} (x)||f_{N,1} (x)| dx \right)^{1/2} \\  \left(\int_I |f_{0,1} (x)||f_{N,1} (x)| dx \right)^{1/2}
 & \ \left(\int_I |f_{N,1} (x)|^2 dx \right)^{1/2}
 \end{array} \right).
\end{equation*}

Since $f_{0,1}$ and $f_{N,1}$ are linearly
 independent, the determinant of the matrix is positive  which
 implies $ \tau > 0$. Taking $A = \sqrt{\tau}\ \|\phi_{N,1}\|_{L_2}
 $ and $B= \sqrt{3}\|\phi_{N,1}\|_{L_2} $,
 it is seen that, for every  $c = \{c_i\} \in l^2$,  $ A \|c\|_{l^2} \leq  \| \sum
 c_i \phi_{N,1}~( \cdot~-~i)\|_{L^2} \leq B \|c\|_{l^2}$. 
 
Further,
 the  functions $\phi_{i,1}$,   $i=1,\ldots,2N-1, i \neq N$, and
 their integer translates are mutually orthogonal. Therefore, the functions $\phi_{i,1}$, $i=1,\ldots,2N-1$, and
 their integer translates form a Riesz basis for $\mathbb{V}_0$.
   \end{proof}

\begin{remark}\label{rem:mra}
By Theorem~\ref{th:mra}, it follows that the set $\{\hat{\phi}_{i,1}
: i=1,2,\ldots,2N-1\} \subset \mathbb{V}_0 $, where $
\hat{\phi}_{i,1}(x) =\phi_{i,1}(x)/\|\phi_{i,1}\|_{L^2}$,
$i=1,2,\ldots,2N-1 $, generates a continuous, compactly supported
multiresolution analysis of $L_2(\mathbb{R})$ by orthonormal
functions.
\end{remark}

\section{Construction of Orthogonal Scaling Functions and Wavelets }\label{sec:OSF}
For generating multiresolution analysis in Theorem~\ref{th:mra}, the
need of orthogonality of functions $\phi_{i,1} \in \mathbb{V}_0$
(c.f.~\eqref{eq:setV0}), $i=1,2,\ldots,2N-1$, requires that CHFIFs
$\tilde{f}_{i,1}$ satisfy conditions~\eqref{eq:risiuij}
and~\eqref{eq:zetaeta}. In this section, as an example,  the values
of parameters in~\eqref{eq:ypoints} are determined so that the
CHFIFs $\tilde{f}_{i,1}$ satisfy these conditions.  For simplicity
in our presentation, such an example is given for $N = 2$ or,
equivalently, when the dimension of vector space $\mathcal{S}_0^1$
is $4$. In addition to requiring that~\eqref{eq:risiuij}
and~\eqref{eq:zetaeta} hold for functions $\tilde{f}_{i,1}$,  $i
=0,1,2,3$, a condition is derived so that the CHFIFs
$\tilde{f}_{0,1}$ and $\tilde{f}_{2,1}$ are also orthogonal. The
orthogonality of the later CHFIFs implies the orthogonality of
$\phi_{1,1}$ with its integer translates that is needed  for the
construction of orthogonal wavelets. Finally, using these explicitly
constructed orthogonal functions $\phi_{i,1}$, $i =1,2,3$,  the
orthogonal wavelets in the space $\mathbb{V}_{-1} \setminus
\mathbb{V}_0$ (c.f.~\eqref{eq:setVk}) are constructed in this
section for $N = 2$.

\newpage
To find the values of parameters in~\eqref{eq:ypoints}, consider
CHFIFs $\tilde{f}_{i,1}$, $i=0,1,2,3$,  as follows:
\begin{enumerate}[(I)]\itemsep1pt
 \item $
\tilde{f}_{0,1}$  corresponds to the data $\big((0,1),
(\frac{1}{2},r_1),(1,0) \big)$,
 \item$ \tilde{f}_{1,1}$ corresponds to the data $\big((0,0),
(\frac{1}{2},1),(1,0) \big)$,
 \item $ \tilde{f}_{2,1}$ corresponds to the data $\big((0,0),
(\frac{1}{2},s_1),(1,1) \big)$,
 \item $\tilde{f}_{3,1}$ corresponds to the data $\big((0,0),
(\frac{1}{2},u_{1,1}),(1,0) \big)$,\\
 where  $r_1$,  $s_1$  and
$u_{1,1}$ are real parameters given by~\eqref{eq:ypoints}.

Further, AFIF $ \tilde{f}_{i,2}$, $i=0,1,2,3$,  are chosen such that
 \item $ \tilde{f}_{i,2}$, $i=0,1,2$ corresponds to the data $\big((0,0),
(\frac{1}{2},0),(1,0) \big)$,
  \item $\tilde{f}_{3,2}$ corresponds to the data $\big((0,0), (\frac{1}{2},1),(1,0) \big)$.
\end{enumerate}

It is observed that $\tilde{f}_{i,1}$, $i=0,1,2$ are self-affine as
$\tilde{f}_{i,2}$, $i=0,1,2$ are zero functions.
By~\eqref{eq:Fncond}, the values of $t_i(x)$, $i=1,2$,  given
by~\eqref{eq:tn}, in the construction of CHFIF $ \tilde{f}_{j,1}$,
$j=0, 1, 2, 3$, are found as the following:
\begin{enumerate}[(i)]\itemsep1pt
\item for CHFIF $
\tilde{f}_{0,1}$, $t_1 (x) = \big( (r_1+ \al_1 -1) x + (1-\al_1), 0
\big) $ and $t_2 (x) = \big( (\al_2-r_1) x + (r_1-\al_2), 0 \big)$
\item  for CHFIF $ \tilde{f}_{1,1}$, $t_1 (x) = \big( x , 0 \big) $  and $t_2(x) = \big( 1- x ,0
\big)$
\item for CHFIF $
\tilde{f}_{2,1}$, $t_1 (x) = \big( (s_1 - \al_1 ) x , 0 \big) $ and
$t_2(x) = \big( ( 1 - s_1 - \al_2  ) x + s_1, 0 \big)$   and
\item for CHFIF $
\tilde{f}_{3,1}$,  $t_1 (x) = \big( u_{1,1} x , x \big) $ and
$t_2(x) = \big( (1 - x)u_{1,1} ,1- x \big)$.
\end{enumerate}

The inner products given by~\eqref{eq:risiuij},~\eqref{eq:zetaeta}
and $\langle \tilde{f}_{0,1}, \tilde{f}_{2,1}\rangle $ are computed
using~\eqref{eq:I}, wherein  the values of $t_i(x)$, $i=1,2$, given
by (i)-(iv) above,  are used. Using~\eqref{eq:risiuij}, the values
of real parameters $r_1$ and $s_1$  given in~\eqref{eq:ypoints} are
computationally obtained as
\begin{equation}\label{eq:r1}
r_1= \frac{4-4 \al_1^2-6 \al_2-2 \al_1 \al_2+3 \al_1^2 \al_2-4
\al_2^2+3 \al_2^3}{4 \left(-4+\al_1^2-\al_1 \al_2+\al_2^2\right)},
\end{equation}
and
\begin{equation}\label{eq:s1}
s_1=\frac{4-6 \al_1-4 \al_1^2+3 \al_1^3-2 \al_1 \al_2-4 \al_2^2+3
\al_1 \al_2^2}{4 \left(-4+\al_1^2-\al_1 \al_2+\al_2^2\right)}.
\end{equation}

Further, the real parameter $u_{1,1}$, for which~\eqref{eq:risiuij}
is satisfied, can be expressed in the form
\begin{equation}\label{eq:u11}
u_{1,1}= \frac{\beta_1 \ \nu_1(\al_1,\al_2,\gm_1,\gm_2) + \beta_2 \
\nu_2(\al_1,\al_2,\gm_1,\gm_2) }{2 \kappa(\al_1,\al_2,\gm_1,\gm_2)}.
\end{equation}

Finally, the inner products $\zeta = \langle
\tilde{f}_{0,1},\tilde{f}_{3,1}\rangle$ and $\eta = \langle
\tilde{f}_{2,1},\tilde{f}_{3,1}\rangle$ can be expressed as
\begin{equation}\label{eq:zeta}
\zeta =  \frac{\beta_1\ \zeta_1(\al_1,\al_2,\gm_1,\gm_2) + \beta_2 \
\zeta_2(\al_1,\al_2,\gm_1,\gm_2) }{\kappa(\al_1,\al_2,\gm_1,\gm_2)}
\end{equation} and
\begin{equation}\label{eq:eta}
\eta =  \frac{\beta_1\ \eta_1(\al_1,\al_2,\gm_1,\gm_2) + \beta_2 \
\eta_2(\al_1,\al_2,\gm_1,\gm_2) }{\kappa(\al_1,\al_2,\gm_1,\gm_2)}.
\end{equation}

Now, using~\eqref{eq:I}  and $t_i(x)$, $i=1,2$,  of CHFIF $
\tilde{f}_{j,1}$, $j=0,2$, given by (I) and (III) above, it is found
that,
\begin{equation}\label{eq:i02temp}
\langle \tilde{f}_{0,1},\tilde{f}_{2,1}\rangle =
\frac{\begin{bmatrix} 4 (r_1+s_1) (2- \al_1\al_2-2\al_1^2-2\al_2^2)
+  8r_1s_1 (4 + \al_1\al_2  - \al_1^2 - \al_2^2) \\-(\al_1^2 +
\al_2^2)^2 + (1 + \al_1^2 + \al_2^2)^3  + 4(\al_1+\al_2)^2 \\ +
\al_1^3(2-2\al_2+6r_1)+ \al_2^3(2-2\al_1 + 6 s_1) \\ -
6r_1\al_1(2-\al_2^2) -
6s_1\al_2(2-\al_1^2)-2(\al_1\al_2+3\al_1+3\al_2)
\end{bmatrix}}{6(2-\al_1-\al_2)(4-\al_1-\al_2)(2-\al_1^2-\al_2^2)}.
\end{equation}

The substitution of values of $r_1$ and $s_1$ from~\eqref{eq:r1}
and~\eqref{eq:s1} in~\eqref{eq:i02temp} gives,
\begin{equation}\label{eq:i02}
\langle  \tilde{f}_{0,1},\tilde{f}_{2,1}\rangle  =
\frac{\rho(\al_1,\al_2)}{12 \left(-4+\al_1^2-\al_1
\al_2+\al_2^2\right) \left(8-6 \al_1+\al_1^2-6 \al_2+2 \al_1
\al_2+\al_2^2\right)}
\end{equation}
where,
\begin{align}\label{eq:rho}
\rho(\al_1,\al_2) &=8+12 \al_1-28 \al_1^2+6 \al_1^3+2 \al_1^4+12
\al_2-14 \al_1 \al_2+18 \al_1^2 \al_2-7 \al_1^3 \al_2 \nno
\\  &\quad \mbox{} -28 \al_2^2+18 \al_1 \al_2^2+6 \al_2^3-7
\al_1 \al_2^3+2 \al_2^4.
\end{align}

Thus, by~\eqref{eq:rho}, $\langle
\tilde{f}_{0,1},\tilde{f}_{2,1}\rangle= 0 $ if and only if
$\rho(\al_1,\al_2)=0$.

To find the orthogonal scaling functions from CHFIF, the free
variables $\al_1, \al_2, \gm_1, \gm_2$  and constrained variables $
\bt_1, \bt_2$  need to be chosen such that $\zeta$,  $\eta$  given
by~\eqref{eq:zeta},~\eqref{eq:eta} respectively satisfy $\zeta=0,\
\eta=0 $ and $ \rho(\al_1,\al_2) =0$. It is found that, with $\al_1
=0$ , $ \al_2 = -3+\sqrt{7}$, $\gm_1 = \frac{-9}{10} $ and $\gm_2 =
\frac{1}{10} \left(-67+29 \sqrt{7}\right)$,
\begin{equation}\label{eq:zee12}
\ \zeta_1(\al_1,\al_2,\gm_1,\gm_2)
\eta_2(\al_1,\al_2,\gm_1,\gm_2))-\eta_1(\al_1,\al_2,\gm_1,\gm_2)\zeta_2(\al_1,\al_2,\gm_1,\gm_2)
= 0\
\end{equation}
and $\ \rho(\al_1,\al_2)=~0$. Writing~\eqref{eq:zeta}
and~\eqref{eq:eta} in matrix form and using~\eqref{eq:zee12}, it is
observed that, there exist infinitely many values of $\bt_1$ and
$\bt_2$ satisfying $\zeta=0,\ \eta=0 $ and $ \rho(\al_1,\al_2) =~0$
with the above values of  $\al_1$ , $ \al_2$ ,   $\gm_1 $ and
$\gm_2$. For example,  one pair of such values of $\bt_1$ and
$\bt_2$ is given by $\beta_1=\frac{1}{20}$ and $\ \beta
_2=\frac{1}{20} \left(3-\sqrt{7}\right)$. Consequently, the
following theorem is obtained from the above analysis:

\begin{theorem}\label{thm:cd}
The functions $ \tilde{f}_{i,1} $, $i=0, 1,2,3$, given by (I)-(IV)
above, are orthogonal CHFIFs in $\mathcal{S}_0^1$ if and only if the
following conditions are satisfied:
\begin{enumerate}[(a)]\itemsep1pt
\item$|\al_1| < 1 , |\al_2| < 1, |\gm_1| < 1 , |\gm_2| < 1,
|\bt_1|+|\gm_1| < 1 $ and $ |\bt_2| + |\gm_2| < 1$
\item  The values of $r_1,s_1$ and $u_{1,1}$ are given by
 \eqref{eq:r1}, \eqref{eq:s1}
 and \eqref{eq:u11}, respectively.
\item The function $\rho(\al_1,\al_2)$, given by~\eqref{eq:rho}, is such that
$\rho(\al_1,\al_2)=0$.
\item The functions $ \zeta$  and $ \eta $, given by~\eqref{eq:zeta}
and~\eqref{eq:eta}, respectively are such that  $ \zeta=0$ and $
\eta=0 $.
 \end{enumerate}
\end{theorem}

\begin{proof}
\begin{enumerate}[(a)]\itemsep0pt
\item The functions $ \tilde{f}_{i,1} $, $i=0, 1,2,3$, given by
(I)-(IV), are CHFIFs if and only if  the identity (a) holds.
\item It follows by the choice of functions as in (i)-(iv) that the
values of $r_1$, $s_1$ and $u_{1,1}$  are given
by~\eqref{eq:r1},~\eqref{eq:s1} and~\eqref{eq:u11} if and only if $\
<\tilde{f}_{0,1},\tilde{f}_{1,1}>  =~0 $, $\
<\tilde{f}_{1,1},\tilde{f}_{2,1}>  =~0 $ and
 $\ <\tilde{f}_{1,1},\tilde{f}_{3,1}>  =~0 $
 respectively.
\item  By~\eqref{eq:i02}, it follows that, $\rho(\al_0,\al_1) = 0$  if and only if  $\
< \tilde{f}_{0,1},\tilde{f}_{2,1}>  = 0 $.
\item  By~\eqref{eq:zeta} and~\eqref{eq:eta}, it is seen that, $ \zeta=0$ and $ \eta=0$
if and only if   $\ < \tilde{f}_{0,1},\tilde{f}_{3,1}>  = 0 $ and $\
< \tilde{f}_{2,1},\tilde{f}_{3,1}>  = 0 $ respectively.
\end{enumerate}
Thus, the functions $ \tilde{f}_{i,1} $, $i=0, 1,2,3$, given by
(I)-(IV) above, are orthogonal functions in $\mathcal{S}_0^1$ if and
only if conditions (a)-(d) are satisfied.  \end{proof}

Having constructed orthogonal functions $ \tilde{f}_{i,1} \in
\mathcal{S}_0^1 $, $i =0,1,2,3$,   that lead to  generation of a
multiresolution analysis of $L_2(\mathbb{R})$ for $N = 2$, by
explicitly determined orthogonal functions $\phi_{i,1}$,  $i=1,2,3$,
it is natural to seek how orthogonal compactly supported continuous
wavelets based on the functions $\phi_{i,1}$, are constructed in
this case. For developing such wavelets based on the vector $\Phi_1
= [ \phi_{1,1} , \phi_{2,1} , \phi_{3,1} ]^t$ of orthogonal scaling
functions and their integer translates, define the wavelet spaces
$\mathbb{W}_k$, $k \in \mathbb{Z}$, as the orthogonal complement of
$\mathbb{V}_k$ (c.f.~\eqref{eq:setVk}) in $ \mathbb{V}_{k-1}$.
Consider the vector  $\Psi_1= [ \psi_{1,1} , \psi_{2,1} , \psi_{3,1}
]^t $, where $\psi_{i,1} \in \mathbb{W}_0$, supported on $[0, 2]$,
are CHFIFs orthogonal to $\phi_{i,1}$, $i=1,2,3$, and their integer
translates. By~\eqref{eq:setVk}, the functions $\psi_{i,1} \in
\mathbb{V}_{k-1} $, $\ i=1,2,3 $, are first components of functions
$\psi_{i} = (\psi_{i,1},\psi_{i,2}) \in \tilde{\mathbb{V}}_{k-1} $,
where $\psi_{i,2}$, $\ i=1,2,3 $,   are AFIFs supported on $[0,2]$.
The set $\{\psi_{1,1}, \psi_{2,1} \psi_{3,1}\} $ is called a `set of
wavelets' associated with the scaling functions $\phi_{1,1} ,
\phi_{2,1} , \phi_{3,1}$  if  $ \{\psi_{i,1}(\cdot -l), i=1,2,3, l
\in \mathbb{Z} \}$ forms a Riesz basis of the space $\mathbb{W}_0$.
The following theorem gives the existence and construction of
non-zero functions $\psi_{i,1}$, $i=1,2,3$, orthogonal to $
\phi_{i,1}$, $i=1,2,3$, and their integer translates:

\begin{theorem}
Let $ \phi_{i,1} \in \mathbb{V}_0 $, $i=1,2,3$, be orthogonal
scaling functions satisfying~\eqref{eq:V0clos}. Then, there exist
non-zero functions $\psi_{i,1} \in \mathbb{W}_0 $,
 orthogonal to $ \phi_{i,1}$, $i=1,2,3$, and their integer
translates.
\end{theorem}

\begin{proof}
Set
\begin{equation}\label{eq:psii1}
 \psi_{i,1} = \left \{
\begin{array}{lll}
g_{1,i,1}(x)& \ \mbox{for}\ & 0 \leq x \leq \frac{1}{2} \\
g_{1,i,2}(x)& \ \mbox{for}\ &  \frac{1}{2} \leq x \leq 1 \\
g_{1,i,3}(x)& \ \mbox{for}\ &  1 \leq x \leq \frac{3}{2} \\
g_{1,i,4}(x)& \ \mbox{for}\ &  \frac{3}{2} \leq x \leq 2 \\
0 & & \mbox{otherwise}
 \end{array} \right.  \quad \quad  i=1,2,3
\end{equation}
and
\begin{equation} \label{eq:psii2}
\psi_{i,2} = \left \{
\begin{array}{lll}
g_{2,i,1}(x)& \ \mbox{for}\ &  0 \leq x \leq \frac{1}{2} \\
g_{2,i,2}(x)& \ \mbox{for}\ &  \frac{1}{2} \leq x \leq 1 \\
g_{2,i,3}(x)& \ \mbox{for}\ &  1 \leq x \leq \frac{3}{2} \\
g_{2,i,4}(x)& \ \mbox{for}\ &  \frac{3}{2} \leq x \leq 2 \\
0 & &\mbox{otherwise}
 \end{array} \right.  \quad \quad  i=1,2,3
\end{equation}
where, CHFIFs $g_{1,i,j}$, $i=1,2,3$,  $j=1,2,3,4$,    corresponds
to the data $\{(x_k, A_{i,l})  : l=2j+k-2,  k=0,1,2\} $,
$A_{i,0}=0=A_{i,8}$ and AFIFs $g_{2,i,j}$, $i=1,2,3$,  $j=1,2,3,4$,
corresponds to the data $\{(x_k, B_{i,l}) :  l=2j+k-2,  k=0,1,2\} $,
$B_{i,0}=0=B_{i,8}$. It is observed that there are 42 unknowns
values, $A_{i,l}$ and $B_{i,l}$, $i=1,2,3$ and $l=1,\ldots,7$. 

The
values of $A_{i,l}$ and $B_{i,l}$, $i=1,2,3$, $l=1,\ldots,7$,  are
found such that
\begin{enumerate}[ (A) ]\itemsep1pt
\item
$\psi_{i,1}$ is orthogonal to $\phi_{j,1}$ and their integer
translates for $i,j=1,2 \ \& \ 3$,
\item $\psi_{i,1}$ is orthogonal
to $\psi_{j,1}$ for $i \neq j,\ i,j=1,2 \ \& \ 3$,
\item $\psi_{i,2}$ is orthogonal to $\phi_{2,2}$ for $ i=1,2 \ \& \ 3$ and
their integer translates  and
\item $\psi_{i,2}$ is orthogonal to $\phi_{j,1}$  for $ i=1,2 \ \& \
3$, $j=1,2$.
\end{enumerate}
These conditions (A)-(D) give rise to a set of 36 non-linear
equations. The functions $\psi_{i,1}$ for $i =1,2\ \& \ 3$ are
normalized to give 3 more equations. Thus, there are 39 non-linear
equations with 42 variables. Hence, there are 3 free variables.
Solving these equations, we get the values of $A_{i,l}$ and
$B_{i,l}\ $,   $\  i=1,2,3$, $\ l=1,\ldots,7$.

Now to show there exist such non-zero functions, let $\al_1 =0$ , $
\al_2 = -3+\sqrt{7}$, $\gm_1 = \frac{-9}{10} $, $\gm_2 =
\frac{1}{10} \left(-67+29 \sqrt{7}\right)$, $\bt_1=\frac{1}{20}$ and
$ \bt_2=\frac{1}{20} \left(3-\sqrt{7}\right)$. Then, the values of
$r_1, s_1$ and $u_{1,1}$ (c.f.~\eqref{eq:r1},~\eqref{eq:s1}
and~\eqref{eq:u11}) for ensuring the orthogonality of $\phi_{i,1}$,
$i=1,2,3$,   are obtained as $ r_1=-3+\sqrt{7}, \ s_1=\frac{1}{6}
\left(-4+\sqrt{7}\right)$ and $u_{1,1}=\frac{-371-40
\sqrt{7}}{70245}$. Further, it is observed that $\zeta_1 =0$ and
$\eta_1 =0$ (c.f.~\eqref{eq:zeta} and~\eqref{eq:eta}).  Therefore, $
\phi_{i,1} \in \mathbb{V}_0 $, $i=1,2,3$,
 are orthogonal scaling functions satisfying~\eqref{eq:V0clos}.  

One possible values of $A_{i,l}$
and $B_{i,l}$, $i=1,2,3$,  and $l=1,\ldots,7$, is obtained as

$A_{1,1}=-1.04784, \ A_{1,2}=0.0125935, A_{1,3}=-1.04663, \
A_{1,4}=0.0231596, $

$ A_{1,5}=0.00599567, \ A_{1,6}=-0.00795969, A_{1,7}=0.00391617, $

$A_{2,1}=0, \ A_{2,2}=0, \ A_{2,3}=-0.298716, \ A_{2,4}=1.32346, \
A_{2,5}=-2.4746, $

$ A_{2,6}=1.12432, \ A_{2,7}=-0.553166, $

$ A_{3,1}=0, \ A_{3,2}=0, \ A_{3,3}=0, A_{3,4}=0, $

$ A_{3,5}=1.06312, \ A_{3,6}=0, \ A_{3,7}=0.983686, $

$B_{1,1}=19.1929,  B_{1,2}=-21.6229, \ B_{1,3}=11.8901, \
B_{1,4}=-11.1171,$

$ B_{1,5}=-4.93066, \ B_{1,6}=1.19803, \ B_{1,7}= 0.567807, $

$ B_{2,1}=0, \ B_{2,2}=0, \ B_{2,3}=0, \ B_{2,4}=0, $

$B_{2,5}=-11.6825,  B_{2,6}=-15.2071, \ B_{2,7}=-3.19525, $

$B_{3,1}=0, \ B_{3,2}=0, \ B_{3,3}=0, \ B_{3,4}=0, $

$B_{3,5}=-13.3015, B_{3,6}=33.9169 \ B_{3,7}=-11.0405$.

Thus, there exist non-zero functions $\psi_{i,1} \in \mathbb{W}_0 $,
 orthogonal to $ \phi_{i,1}$, $i=1,2,3$, and their integer
translates.  \end{proof}

The graphs of $\ \psi_{i,1} $ for $k=1,2,3$ corresponding to the
values of $A_{i,l}$ and $B_{i,l}$ given above is shown in the
Figure~\ref{fig:wv}.

\begin{figure}[hbt]
\centering \subfigure[\textbf{$\ \psi_{1,1}
$}]{\epsfig{file=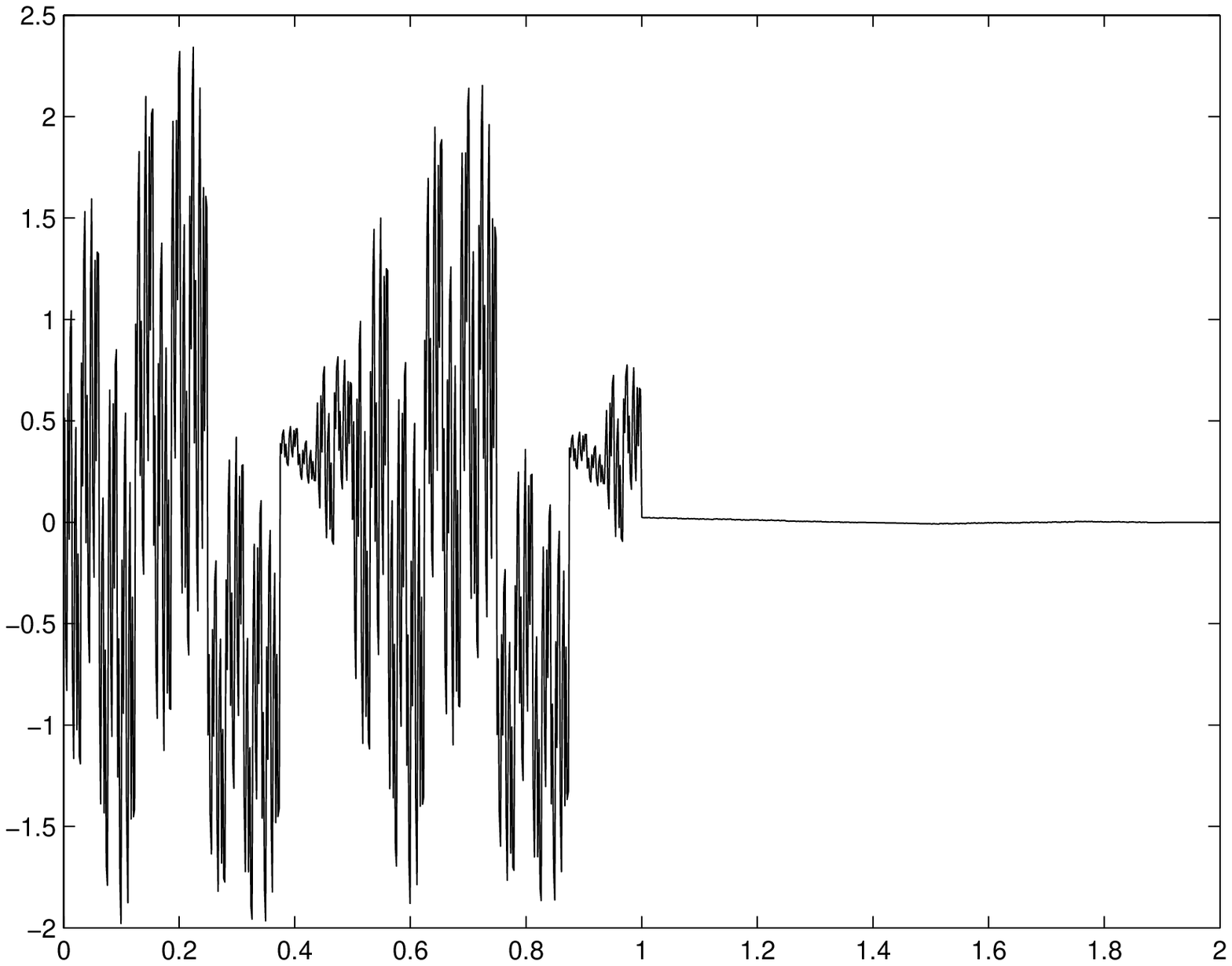, width=3.2cm}\label{fig:ww1}}\hspace{1cm}
\subfigure[\textbf{$\ \psi_{2,1} $}]{\epsfig{file=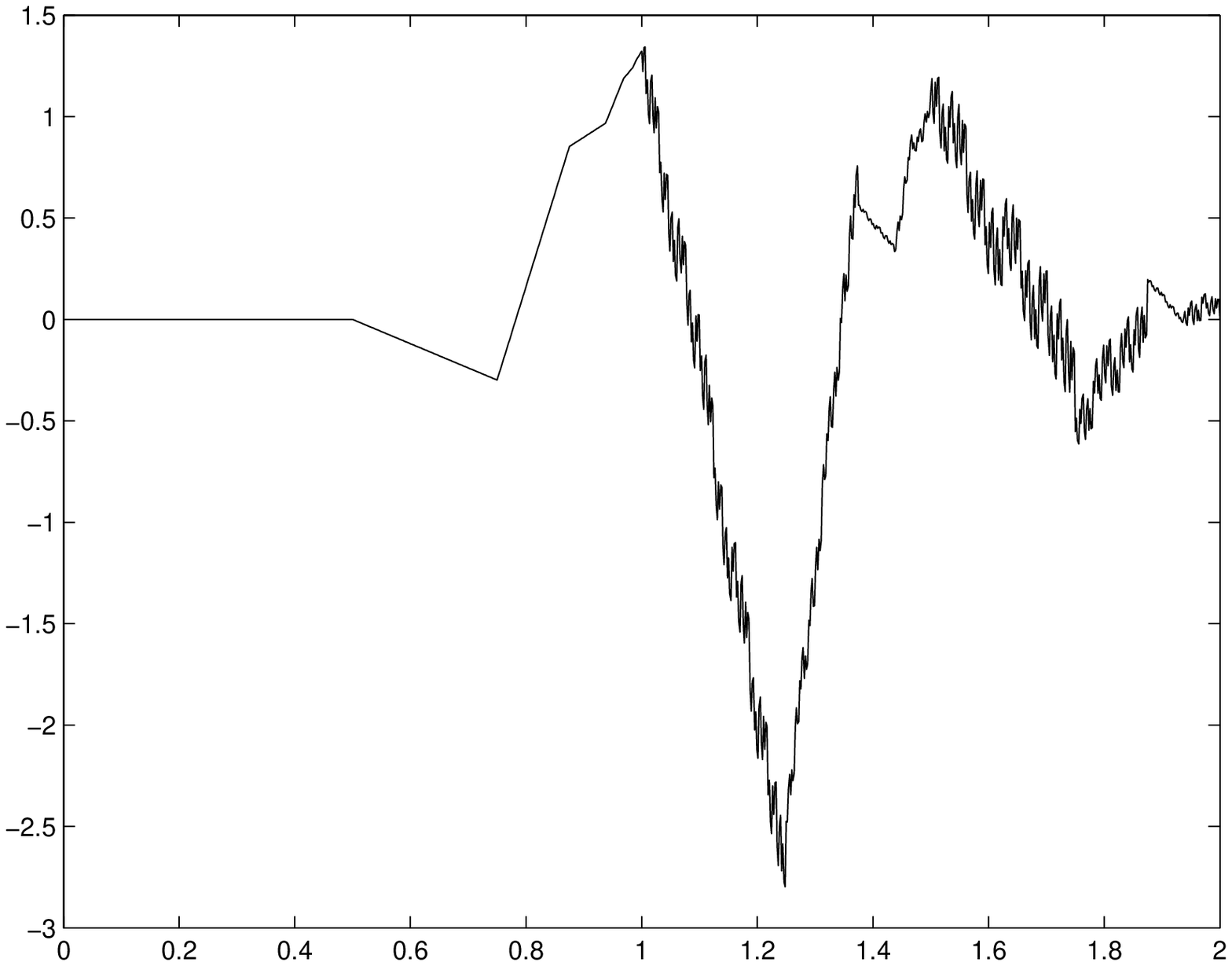,
width=3.2cm}\label{fig:ww2}}\hspace{1cm} \subfigure[\textbf{$\
\psi_{3,1} $}]{\epsfig{file=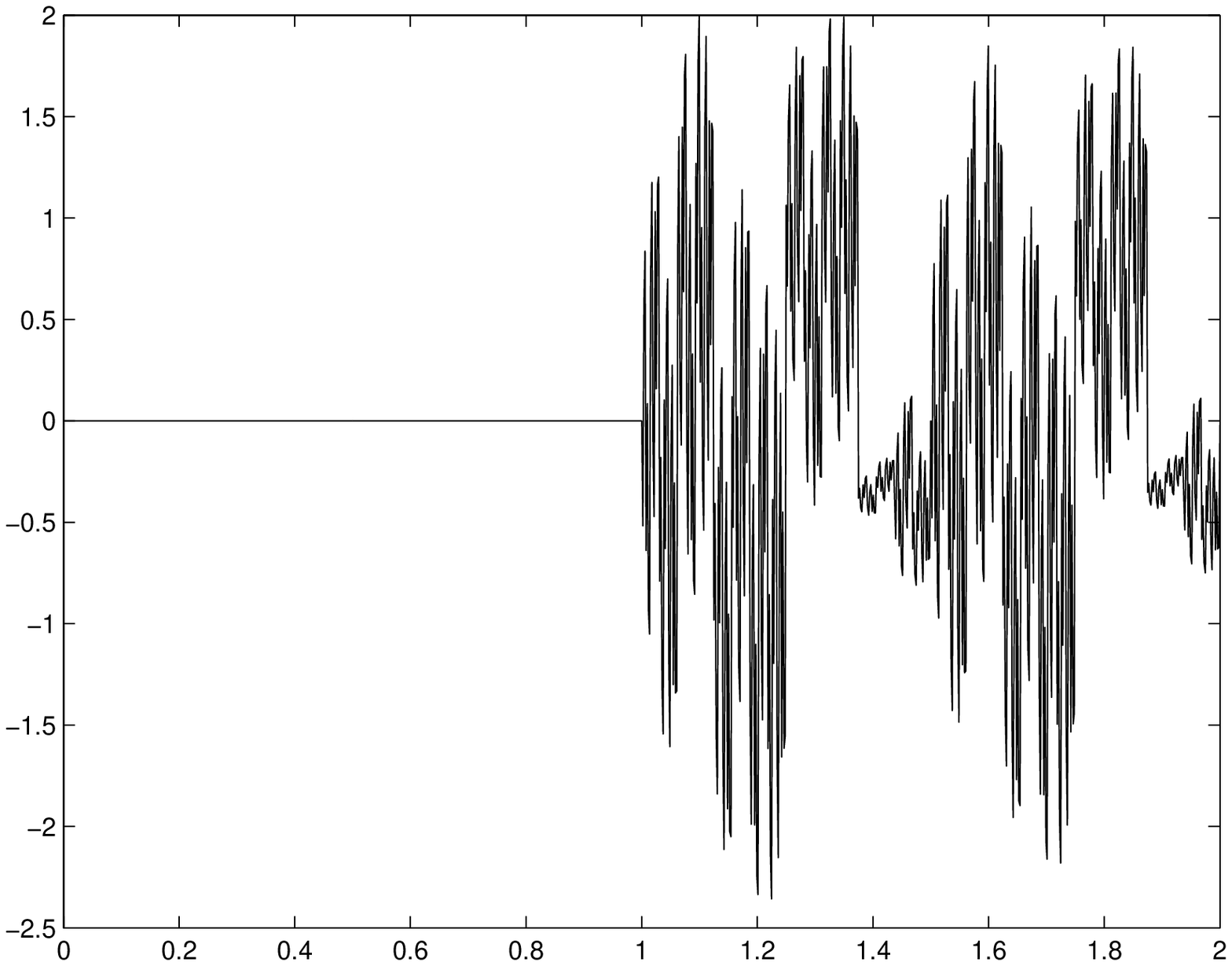,
width=3.2cm}\label{fig:ww3}}\caption{Wavelets $\psi_{i,1}$ for $i
=1,2\ \& \ 3$ \label{fig:wv}}
\end{figure}
 
In general, if $ \phi_{i,1}  $, $i=1,\ldots,2N-1$,  are orthogonal
scaling functions satisfying~\eqref{eq:V0clos}, then there exist
non-zero functions  $\psi_{i,1} \in \mathbb{W}_0 $,
$i=1,\ldots,(2N-1)(N-1)$ orthogonal to $ \phi_{i,1}$,
$i=1,\ldots,2N-1$, and their integer translates.

\section{Conclusions}
In this paper, multiresolution analysis arising from Coalescence
Hidden-variable Fractal Interpolation Functions is accomplished. The
availability of a larger set of free variables and constrained
variables with CHFIF in multiresolution analysis based on CHFIFs
provides more control in reconstruction of functions in
$L_2(\mathbb{R})$ than that provided by multiresolution analysis
based only on affine FIFs. In our approach, the vector space of
CHFIFs is introduced, its dimension is determined and Riesz bases of
vector subspaces $\mathbb{V}_k, k \in \mathbb{Z}$, consisting of
certain CHFIFs in $L_2(\mathbb{R})\bigcap C_0(\mathbb{R})$ are
constructed. As a special case, for the vector space of CHFIFs of
dimension $4$, orthogonal bases for the vector subspaces
$\mathbb{V}_k, k \in \mathbb{Z}$, are explicitly constructed and,
using these bases, compactly supported continuous orthonormal
wavelets are generated.

\section*{Acknowledgement}
The second author thanks CSIR for Research Grant No:
9/92(417)/2005-EMR-I for the present work.


\begin{thebibliography}{10}


\bibitem{chand07}
{Chand A.K.B.} and {Kapoor G.P.}
\newblock Smoothness analysis of coalescence hidden variable fractal
  interpolation functions.
\newblock {\em International Journal of Non-Linear Science}, 3:15--26,~2007.

\bibitem{donovan96}
{George C. Donovan}, {Jeffrey S. Geronimo}, {Douglas P. Hardin}, and {Peter R.
  Massopust}.
\newblock Construction of orthogonal wavelets using fractal interpolation
  function.
\newblock {\em SIAM Journal on Mathematical Analysis}, 27(4):1158--1192,~1996.

\bibitem{geronimo93}
{Geronimo J.S.} and {Hardin D.P.}
\newblock Fractal interpolation surfaces and a related 2-{D} multiresolution
  analysis.
\newblock {\em Journal of Mathematical Analysis and Applications},
  176:561--586,~1993.

\bibitem{geronimo92}
{Geronimo J.S.}, {Hardin D.P.}, and {Massopust P.R.}
\newblock {\em An Application of Coxeter groups to the construction of wavelet
  bases in $\mathbb{R}^n$, Lecture Notes in Pure and Applied Mathematics, 157:
  Contempory Aspects of Fourier Analysis}.
\newblock {Bray W.U.} and {Milojevic P.S.} and {Stanojevic} ed., Marcel Dekker,
  New York, Basel,~1992.

\bibitem{geronimo94}
{Geronimo J.S.}, {Hardin D.P.}, and {Massopust P.R.}
\newblock Fractal functions and wavelet expansions based on several scaling
  functions.
\newblock {\em Journal of Approximation Theory}, 78:373--401,~1994.

\bibitem{goodman94}
{Goodman T.N.T.} and {Lee S.L.}
\newblock Wavelets of multiplicity r.
\newblock {\em Transactions of American Mathematical Society}, 342:307--324,~1994.

\bibitem{grochenig92}
{Gr\"ochenig K.} and {Madych W.R.}
\newblock Multiresolution analysis, haar bases and self-similar tilings of
  $\mathbb{R}^n$.
\newblock {\em IEEE Transactions On Information Theory}, 38(2):556--568,~1992.

\bibitem{hardin92}
{Hardin D.P.}, {Kessler B.}, and {Massopust P.R.}
\newblock Multiresolution analysis based on fractal functions.
\newblock {\em Journal of Approximation Theory}, 71:104--120,~1992.

\bibitem{herve94}
{Herve L.}
\newblock Multiresolution analysis of multiplicity d: Applications to dyadic
  interpolation.
\newblock {\em Applied Computational Harmonic Analysis}, 1:299--315,~1994.

\bibitem{mallat89_1}
{Mallat S}.
\newblock Multiresolution approximations and wavelet orthonormal bases of
  $l^2(\mathbb{R})$.
\newblock {\em Transactions of American Mathematical Society}, 315:69--87,~1989.

\bibitem{mallat89_2}
{Mallat S}.
\newblock A theory for multiresolution signal decomposition: The wavelet
  representation.
\newblock {\em IEEE Transactions on Pattern Analysis and Machine Intelligence},
  11:674--693,~1989.

\bibitem{massopust10}
{Massopust P.}
\newblock {\em Interpolation and Approximation with Splines and Fractals}.
\newblock Oxford University Press,~2010.

\bibitem{meyer89}
{Meyer I.}
\newblock {\em Ondelettes et Operateurs}.
\newblock Hermann, Paris,~1989.

\bibitem{telesca04}
{Telesca L.}, {Lapenna V.}, and {Alexis N.}
\newblock Multiresolution wavelet analysis of earthquakes.
\newblock {\em Chaos, Solitons and Fractals}, 22(3):741--748,~2004.

\end{thebibliography}
\end{document}